\documentclass[12pt]{article}

\usepackage{graphicx}

\usepackage{bbm}

\usepackage{amsmath}
\usepackage{amsthm}
\usepackage{amsfonts}
\usepackage{amssymb}
\usepackage{color}
\usepackage{comment}

\newcommand{\me}{\mathbb{E}}
\newcommand{\mn}{\mathbb{N}}
\newcommand{\mr}{\mathbb{R}}
\DeclareMathOperator{\1}{\mathbbm{1}}

\newcommand{\eee}{{\rm e}}
\newcommand{\dd}{{\rm d}}

\newcommand{\mmp}{\mathbb{P}}

\newtheorem{thm}{Theorem}[section]
\newtheorem{lemma}[thm]{Lemma}

\theoremstyle{definition}

\theoremstyle{remark}

\begin{document}
\title{A law of the iterated logarithm for iterated random walks, with application to random recursive trees}
\date{}
\author{Alexander Iksanov\footnote{Faculty of Computer Science and Cybernetics, Taras Shevchenko National University of Kyiv, Ukraine; e-mail address:
iksan@univ.kiev.ua} \ \ Zakhar Kabluchko\footnote{Institut f\"{u}r Mathematische Statistik, Westf\"{a}lische Wilhelms-Universit\"{a}t M\"{u}nster,
48149 M\"{u}nster, Germany; e-mail address: zakhar.kabluchko@uni-muenster.de} \ \ and \ \ Valeriya Kotelnikova\footnote{Faculty of Computer Science and Cybernetics, Taras Shevchenko National University of Kyiv, Ukraine; e-mail address:
valeria.kotelnikova@unicyb.kiev.ua}}

\maketitle

\begin{abstract}
Consider a Crump-Mode-Jagers process generated by an increasing random walk whose increments have finite second moment. Let $Y_k(t)$ be the number of individuals in generation $k\in \mathbb N$ born in the time interval $[0,t]$.
We prove a law of the iterated logarithm for $Y_k(t)$ with fixed $k$, as $t\to +\infty$.  As a consequence, we derive a law of the iterated logarithm for the number of vertices at a fixed level $k$ in a random recursive tree, as the number of vertices goes to $\infty$.
\end{abstract}

\section{A law of the iterated logarithm for iterated random walks}\label{sec:main}

Let $\xi_1$, $\xi_2,\ldots$ be independent copies of an almost surely (a.s.) positive random variable $\xi$. Denote by $S:=(S_n)_{n\in\mn}$ the {\it standard random walk} with increments $\xi_n$ for $n\in\mn$, that is,  $S_n:=\xi_1+\ldots+\xi_n$ for $n\in\mn$. The corresponding {\it renewal process} $(Y(t))_{t\geq 0}$ is defined by
$$Y(t):=\sum_{n\ge 1} \1_{\{S_n\le t\}}, \quad t\geq 0.$$ Put $V(t):=\me Y(t)$ for $t\geq 0$. The function $V$ is called {\it renewal function}.

Now we recall the construction of a general branching process (a.k.a.\ Crump-Mode-Jagers process) generated by $S$. There is a population of individuals initiated at time $0$ by one
individual, the ancestor. An individual born at time $t\geq 0$ produces offspring whose birth times have the same
distribution as $(t + S_n)_{n\in\mn}$. All individuals act independently of each other. For $k\in\mn$, an individual resides in the $k$th
generation if it has exactly $k$ ancestors. For $k\in\mn$ and $t\geq 0$, denote by $S^{(k)}$ the collection of the birth times in
the $k$th generation and by $Y_k(t)$ the number of the $k$th generation individuals with birth times $\leq t$. Put $V_k(t):=\me Y_k(t)$. Plainly, $Y_1(t)=Y(t)$ and $V_1(t)=V(t)$ for $t\geq 0$. Following \cite{Bohun etal:2022, Iksanov+Rashytov+Samoilenko:2023} we call the sequence $(S^{(k)})_{k\geq 2}$ an {\it iterated standard random walk}.

Let $(\xi_1,\eta_1)$, $(\xi_2,\eta_2),\ldots$ be independent copies of an $\mr^2$-valued random vector $(\xi, \eta)$ with positive arbitrarily dependent components. Put
\begin{equation*}
T_n:= S_{n-1}+ \eta_n,\quad n\in\mn.
\end{equation*}
The sequence $T:=(T_n)_{n\in\mn}$ is called {\it (globally) perturbed random walk}. A counterpart of $(S^{(k)})_{k\geq 2}$, obtained by replacing in the aforementioned construction $S$ with $T$, is called an {\it iterated perturbed random walk}.

Now we review some previous work on iterated random walks and more general processes. Denote by $D$ the Skorokhod space of c\`{a}dl\`{a}g functions on $[0,\infty)$.
\begin{itemize}
\item Iterated standard random walks.

\noindent - Theorem 1.3 in \cite{Iksanov+Kabluchko:2018a} is a functional central limit theorem (FCLT) for $(Y_k)_{k\in\mn}$, properly scaled, normalized and centered, on $D^\mn$ equipped with the product $J_1$-topology, under the assumption $\me\xi^2<\infty$.

\noindent - Theorem 2.6 in \cite{Iksanov+Marynych+Rashytov:2022} derives the asymptotics of ${\rm Var}\,Y_k(t)$ as $t\to\infty$ when $k\in\mn$ is fixed under the assumptions that $\me\xi^2<\infty$ and that the distribution of $\xi$ is nonlattice (see Section \ref{sect:aux} for the definition).

\noindent - Theorems 2.1 and 2.2 in \cite{Iksanov+Kabluchko:2018b} prove weak convergence of finite-dimensional distributions $(Y_{\lfloor k(t)u\rfloor}(t))_{u>0}$, properly normalized and centered, under the assumption that the distribution of $\xi$ is exponential. Here, $k$ is a positive function satisfying $k(t)\to +\infty$ and $k(t)=o(t)$ as $t\to\infty$.

\item Iterated perturbed random walks.

\noindent -Theorem 2.8 in \cite{Iksanov+Rashytov+Samoilenko:2023} is a FCLT for a counterpart of $(Y_k)_{k\in\mn}$, properly scaled, normalized and centered, on $D^\mn$ equipped with the product $J_1$-topology, under the assumptions $\me\xi^2<\infty$ and $\me \eta^a<\infty$ for some $a>0$. Also, this paper proves the elementary renewal theorem and its refinement, Blackwell's theorem and the key renewal theorem for a counterpart of $V_k$ with $k\in\mn$ fixed.

\noindent - The paper \cite{Bohun etal:2022} proves the elementary renewal theorem, Blackwell's theorem and the key renewal theorem for a counterpart of $V_{k(t)}(t)$ under various assumptions imposed on the distributions of $\xi$ and $\eta$. Here, $k$ is an integer-valued function satisfying $k(t)\to +\infty$ and $k(t)=o(t^{2/3})$ as $t\to\infty$. The most interesting observation of the cited paper is that, under the assumptions that $\me\xi^3<\infty$, that $\me \eta^2<\infty$ and that the distribution of $\xi$ is spread out, the asymptotics of $V_{k(t)}(t)$ exhibits a phase transition at generations $k(t)$ of order $t^{1/2}$.

\noindent - Proposition 3.1 and Theorems 3.2 and 3.3 in \cite{Buraczewski+Dovgay+Iksanov:2020} prove weak convergence of the finite-dimensional distributions for a counterpart of $(Y_{\lfloor k(t)u\rfloor}(t))_{u>0}$, properly normalized and centered, under the assumptions $\me \xi^2<\infty$, $\me \eta<\infty$ and that the distribution of $\xi$ is nonlattice. Here, $k$ is a positive function satisfying $k(t)\to +\infty$ and $k(t)=o(t^{1/3})$ as $t\to\infty$.

\noindent - Theorems 3.1 and 3.2 in \cite{Iksanov+Marynych+Samoilenko:2023} provide an improvement over the aforementioned result from \cite{Buraczewski+Dovgay+Iksanov:2020}, in which the distribution of $\xi$ is not required to be nonlattice and, more importantly, the condition $k(t)=o(t^{1/3})$ as $t\to\infty$ is replaced with $k(t)=o(t^{1/2})$.

\noindent - Theorem 1 in \cite{Iksanov+Marynych+Rashytov:2022} proves weak convergence of the finite-dimensional distributions for a counterpart of $(Y_{\lfloor k(t)u\rfloor}(t))_{u>0}$, properly normalized and centered, under the assumptions that $\me\xi^2=\infty$, that the distribution of $\xi$ belongs to the domain of attraction of an $\alpha$-stable distribution, $\alpha\in (1,2]$ and that $\me \min (\eta, t)=O(t^{2-\gamma})$ as $t\to\infty$ for some $\gamma\in (1,2)$ specified in the paper. Here, $k$ is a positive function satisfying $k(t)\to +\infty$ and $k(t)=o(t^{(\gamma-1)/2})$ as $t\to\infty$.

\item More general iterated sequences.

\noindent - Theorem 3.2 in \cite{Gnedin+Iksanov:2020} is a FCLT for a counterpart of $(Y_k)_{k\in\mn}$, properly scaled, normalized and centered, on $D^\mn$ equipped with the product $J_1$-topology, under the assumption that an appropriate FCLT holds for $Y_1$ and some further assumptions.

\end{itemize}

For a family $(x_t)$ of real numbers we write $C((x_t))$ for the set of its limit points. Recall that $0!=1$. Now we state a law of the iterated logarithm for $Y_k$ for fixed $k\in\mn$.
\begin{thm}\label{main}
Assume that $\sigma^2:={\rm Var}\,\xi\in (0, \infty)$. Then, for each fixed $k\in\mn$,
$$C\bigg(\bigg(\frac{a_k \big(Y_k(t)-t^k/(k!
\mu^k)\big)}{(2t^{2k-1}\log\log t)^{1/2}}: t>\eee\bigg)\bigg)=[-1,1]\quad\text{{\rm a.s.}},$$ where $$a_k:=\sigma^{-1} \mu^{k+1/2}(k-1)!(2k-1)^{1/2}$$ and $\mu:=\me\xi<\infty$. In
particular,
$${\lim\sup}_{t\to\infty}\frac{a_k \big(Y_k(t)-t^k/(k!
\mu^k)\big)}{(2t^{2k-1}\log\log t)^{1/2}}=1\quad\text{{\rm a.s.}}$$ and $${\lim\inf}_{t\to\infty}\frac{a_k \big(Y_k(t)-t^k/(k!
\mu^k)\big)}{(2t^{2k-1}\log\log t)^{1/2}}=-1\quad\text{{\rm a.s.}}$$
The centering $t^k/(k!\mu^k)$ can be replaced with $\me Y_k(t)$ everywhere.
\end{thm}

\section{Application to random recursive trees}
In this section we state a law of the iterated logarithm for the profile of the random recursive tree (RRT) and prove it using Theorem~\ref{main} in the special case when the random variable $\xi$ has an exponential distribution of unit mean. For our purposes, the following continuous-time construction of the RRT is convenient (see, e.g., Example~6.1 in~\cite{holmgren+janson:2017}).  At time $0$, the RRT consists of $1$ vertex, the root, located at level $0$. This vertex generates offspring at arrival times of a unit intensity Poisson process. These offspring are located at level $1$. More generally, each vertex of the tree, immediately after its birth, starts to generate offspring at rate $1$, and all vertices act independently. If some vertex is located at level $k$, then its offspring appear at level $k+1$, so that the level of any vertex is its distance to the root. Clearly, one can identify the birth times of the vertices at level $k\in \mn$ with the process $S^{(k)}$, as defined in Section~\ref{sec:main}, with the random variable $\xi$ having an exponential distribution of unit mean. Let $\tau_1 < \tau_2 < \ldots $  be the birth times of the vertices of the RRT, excluding the root born at time $\tau_0=0$. For $n\in\mn$, at time $\tau_n$, the tree consists of $n+1$ vertices. For $k\in \mn$, let $X_n(k)= Y_k(\tau_n)$ be the number of vertices in this tree having distance $k$ to the root at time $\tau_n$. The function $k\mapsto X_n(k)$ is called the profile of the RRT. Its asymptotic behavior as $n\to\infty$ has been much studied. For example, a central limit theorem for $X_n(k)$ with fixed $k$ has been obtained in~\cite{Fuchs+Hwang+Neininger:2006}; see also~\cite{Iksanov+Kabluchko:2018a} for a functional version.
As a corollary of Theorem~\ref{main} we shall prove the following law of the iterated logarithm for $X_n(k)$.

%B.\ Pittel, see %p.~339 in \cite{Pittel:1994}

\begin{thm}\label{RRT}
For each fixed $k\in \mn$,
$$
C\bigg(\bigg(\frac{(k-1)!(2k-1)^{1/2} \big(X_n(k)-(\log n)^k/k!\big)}{(2(\log n)^{2k-1}\log\log\log n)^{1/2}}: n > \eee^\eee \bigg)\bigg)
=
[-1,1]\quad\text{{\rm a.s.}}
$$
\end{thm}
For $k=1$, the claim is known (see Theorem 3' in \cite{renyi:1962}) since the sequence $(X_n(1))_{n\in\mn}$ has the same joint distribution as $(B_{1} + \ldots + B_n)_{n\in \mn}$, where $B_1,B_2,\ldots$ are independent Bernoulli random variables with $\mmp\{B_k = 1\}  = 1/k$.
\begin{proof}[Proof of Theorem~\ref{RRT}]
Let $\xi$ be a random variable having an exponential distribution of unit mean. Then $a_k=(k-1)!(2k-1)^{1/2}$ since $\mu=\sigma^2=1$, and Theorem~\ref{main} takes the form
\begin{equation}\label{eq:LIL_CMJ_exponential}
C\bigg(\bigg(\frac{a_k\big(Y_k(t)-t^k/k!\big)}{(2 t^{2k-1}\log\log t)^{1/2}}: t>\eee \bigg)\bigg)
=
[-1,1]\quad\text{{\rm a.s.}}
\end{equation}
For $t\geq 0$, let $n(t)\in \{0,1,\ldots\}$ be the unique index with $\tau_{n(t)} \leq t < \tau_{n(t) +1}$. Then, $(n(t)+1)_{t\geq 0}$ is the Yule process for which it is known (see Theorems 1 and 2 on pp.~111-112 in \cite{athreya_ney_book}) that $\lim_{t\to\infty} \eee^{-t} n(t)= W$ a.s., where $W$ is a random variable satisfying $W>0$ a.s. It follows that $\lim_{t\to\infty}(\log n(t) - t)= \log W$ a.s.\ and $\lim_{t\to\infty} t^{-1}\log n(t)=1$ a.s. Consequently, $$
\frac{t^k - (\log n(t))^k}{t^{k-1}} =
(t -\log n(t))\cdot \frac{\sum_{j=0}^{k-1} t^{j} (\log n(t))^{k-1-j}}{t^{k-1}}
\overset{}{\underset{t\to\infty}\longrightarrow}
 - k\cdot \log W
\quad
\text{{\rm a.s.}}
$$
Note that $Y_k(\tau_{n(t)}) = Y_k(t)$. The identity
\begin{multline*}
\frac{Y_k(\tau_{n(t)})-(\log n(t))^k/k!}{(2(\log n(t))^{2k-1}\log\log\log n(t))^{1/2}}
=
\left(\frac{Y_k(t)-t^k/k!}{(2 t^{2k-1}\log\log t)^{1/2}}  +  \frac{t^k -(\log n(t))^k}{k!(2 t^{2k-1}\log\log t)^{1/2}}\right)
\\
\times
\frac{(2 t^{2k-1}\log\log t)^{1/2}}{(2(\log n(t))^{2k-1}\log\log\log n(t))^{1/2}},
\end{multline*}
in which
$$
\lim_{t\to\infty} \frac{(2 t^{2k-1}\log\log t)^{1/2}}{(2(\log n(t))^{2k-1}\log\log\log n(t))^{1/2}}
=1\quad\text{and}\quad \lim_{t\to\infty} \frac{t^k -(\log n(t))^k}{k!(2 t^{2k-1}\log\log t)^{1/2}}=0 
\quad
\text{{\rm a.s.}},
$$
combined with~\eqref{eq:LIL_CMJ_exponential} entails that
$$
C\bigg(\bigg(\frac{a_k \big(Y_k(\tau_{n(t)})-(\log n(t))^k/k!\big)}{(2(\log n(t))^{2k-1}\log\log\log n(t))^{1/2}}: t\ge \tau_{\lceil \eee^{\eee}\rceil}\bigg)\bigg)=[-1,1]\quad\text{{\rm a.s.}}
$$
It follows that
$$
C\bigg(\bigg(\frac{a_k\big(Y_k(\tau_n)-(\log n)^k/k!\big)}{(2(\log n)^{2k-1}\log\log\log n)^{1/2}}: n > \eee^\eee \bigg)\bigg)
=
[-1,1]\quad\text{{\rm a.s.}}
$$
This completes the proof of Theorem~\ref{RRT} since $X_n(k)= Y_k(\tau_n)$.
\end{proof}

\section{Auxiliary results}\label{sect:aux}

For the proof of Theorem~\ref{main} we shall need the following strong approximation result, which follows, for instance, from Theorem 12.13 on p.~227 in \cite{Kallenberg:1997}.
\begin{lemma}\label{chs}
Assume that $\sigma^2={\rm Var}\,\xi\in (0,\infty)$. Then there exists
a standard Brownian motion $W$ such that
$$\lim_{t\to\infty}\frac{\sup_{0\leq s\leq
t}\,\big|Y(s)-V(s)-\sigma\mu^{-3/2}W(s)\big|}{(t\log\log t)^{1/2}}=0\quad\text{{\rm
a.s.}},$$ where $\mu=\me\xi<\infty$.
\end{lemma}

Now we lay down the ground for the subsequent proofs.
For $t\geq 0$, $k\geq 2$ and $r\in\mn$, let $Y_{k-1}^{(r)}(t)$ be the number of successors in the $k$th generation with birth times within $[S_r,\,S_r+t]$ of the first generation individual with birth time $S_r$. Then
\begin{equation}\label{eq:recur}
Y_k(t)=\sum_{r\geq 1} Y_{k-1}^{(r)}(t-S_r)\1_{\{S_r\le t\}}.
\end{equation}
By the branching property, $(Y_{k-1}^{(1)}(t))_{t\ge 0}, (Y_{k-1}^{(2)}(t))_{t\ge 0}, \ldots$ are independent copies of $(Y_{k-1}(t))_{t\geq 0}$ which are also independent of $S$.
Passing in \eqref{eq:recur} to expectations we obtain, for $k\geq 2$ and $t\geq 0$,
\begin{equation}\label{eq:recurV}
V_k(t)=\int_{[0,\,t]} V_{k-1}(t-y)\dd V(y)=\int_{[0,\,t]} V(t-y)\dd V_{k-1}(y).
\end{equation}
Thus, $V_k$ is the $k$-fold Lebesgue-Stieltjes convolution of $V$ with itself. 	

For fixed $d>0$, the distribution of a positive random variable is called $d$-lattice if it is concentrated on the lattice $(nd)_{n\in\mn_0}$ and not concentrated on $(nd_1)_{n\in\mn_0}$ for any $d_1>d$. The number $d$ is called span of the corresponding lattice distribution. The distribution of a positive random variable is called nonlattice if it is not $d$-lattice for any $d>0$. Lemma~\ref{lem:vk} collects some properties of $V_k=\me Y_k$.
\begin{lemma}\label{lem:vk}
Fix any $k\in\mn$.

\noindent (a) Assume that $\mu=\me\xi<\infty$. Then $$\lim_{t\to\infty}\frac{V_k(t)}{t^k}=\frac{1}{k!\mu^k}.$$

\noindent (b) Assume that $\mu=\me\xi<\infty$. Then $$\lim_{t\to\infty}\frac{V_k(t+h)-V_k(t)}{t^{k-1}}=\frac{h}{(k-1)!\mu^k}$$ for each $h>0$ if the distribution of $\xi$ is nonlattice and $h=id$, $i\in\mn$ if the distribution of $\xi$ is $d$-lattice.

\noindent (c) Assume that
$\me\xi^2<\infty$. Then $$-\infty<{\lim\inf}_{t\to\infty}\frac{V_k(t)-t^k/(k!\mu^k)}{t^{k-1}}\leq {\lim\sup}_{t\to\infty}\frac{V_k(t)-t^k/(k!\mu^k)}{t^{k-1}}<\infty.$$

\noindent (d) For all $x,h\geq 0$,
\begin{equation}\label{eq:subad}
V_k(x+h)-V_k(x)\leq (V(h)+1)(V(x+h))^{k-1}.
\end{equation}
\end{lemma}

\begin{proof}
(a) See, for instance, Theorem 1.16 on p.~38 in \cite{Mitov+Omey:2014}.

\noindent (b)
When the distribution of $\xi$ is nonlattice, this is a particular case ($\eta=\xi$) of Theorem~2.4 in \cite{Iksanov+Rashytov+Samoilenko:2023}. Assume now that, for some $d>0$, the distribution of $\xi$ is $d$-lattice. Then  $\lim_{t\to\infty}(V_1(t+h)-V_1(t))=\mu^{-1}h$ for $h=id$, $i\in\mn$ by Blackwell's theorem, see Theorem~1.10 in \cite{Mitov+Omey:2014}. With this at hand, the same proof by induction as in \cite{Iksanov+Rashytov+Samoilenko:2023} also works in the lattice case.

\noindent (c) Assume that the distribution of $\xi$ is nonlattice. Using Theorem 2.2 in \cite{Iksanov+Rashytov+Samoilenko:2023}, with $\eta=\xi$ in the notation of that paper, we conclude that, for each fixed $k\in\mn$,
$$V_k(t)-\frac{t^k}{k!\mu^k}~\sim~\frac{bk t^{k-1}}{(k-1)!\mu^{k-1}},\quad t\to\infty,$$ where $b=\me \xi^2/(2\mu^2)-1\in\mr$.

Assume now that, for some $d>0$, the distribution of $\xi$ is $d$-lattice. Then for each $t\geq 0$ there exists $n\in\mn_0$ such that $t\in[nd,(n+1)d)$. Hence, by monotonicity,
\begin{multline*}
	{\lim\sup}_{t\to\infty}\frac{V_k(t)-t^k/(k!\mu^k)}{t^{k-1}}\le {\lim\sup}_{n\to\infty}\frac{V_k((n+1)d)-(nd)^k/(k!\mu^k)}{(nd)^{k-1}} \\= {\lim\sup}_{n\to\infty}\frac{V_k((n+1)d)-V_k(nd)}{(nd)^{k-1}}+ {\lim\sup}_{n\to\infty}\frac{V_k(nd)-(nd)^k/(k!\mu^k)}{(nd)^{k-1}}.
\end{multline*}
The former limit superior (actually, the full limit) is finite according to Lemma \ref{lem:vk}(b), and the latter limit superior (the full limit) is finite according to Lemma \ref{lattice} with $\eta=\xi$. The finiteness of the lower limit follows analogously.

\noindent (d) We use mathematical induction in $k$. If $k=1$, then \eqref{eq:subad} expresses a known fact that the renewal function $V+1$ is subadditive, see, for instance, Theorem 1.7 on p.~10 in~\cite{Mitov+Omey:2014}. Assume that \eqref{eq:subad} holds for $k=j$ and note that, by monotonicity and \eqref{eq:recurV}, $V_j(h)\leq (V(h))^j \leq (V(h)+1) (V(x+h))^{j-1}$ for $x, h\geq 0$. We obtain, by another application of \eqref{eq:recurV},
\begin{multline*}
V_{j+1}(x+h)-V_{j+1}(x)=\int_{[0,\,x]}(V_j(x+h-y)-V_j(x-y)){\rm d}V(y)+\int_{(x,\,x+h]}V_j(x+h-y){\rm d}V(y)\\\leq (V(h)+1)\int_{[0,\,x]}(V(x+h-y))^{j-1}{\rm d}V(y)+V_j(h)(V(x+h)-V(x))\\\leq (V(h)+1) (V(x+h))^{j-1} V(x)+(V(h)+1) (V(x+h))^{j-1}(V(x+h)-V(x))\\=(V(h)+1) (V(x+h))^j.
\end{multline*}
\end{proof}

Here is another important ingredient for our proof of
Theorem \ref{main}.
\begin{lemma}\label{center:mom}
Fix any $k\in\mn$. Assume that ${\rm Var}\,\xi\in (0, \infty)$. Then
\begin{equation}\label{eq:4th}
\me \sup_{0\leq s\leq t} (Y_k(s)-V_k(s))^2=O(t^{2k-1}),\quad t\to\infty.
\end{equation}
\end{lemma}
\begin{proof}
In the setting of iterated perturbed random walks a counterpart of \eqref{eq:4th} is a consequence of Lemmas 4.2(b) and 3.1(c) in \cite{Gnedin+Iksanov:2020} under the assumptions $\me\xi^2<\infty$ and $\me\eta<\infty$. Relation \eqref{eq:4th} itself follows on putting $\eta=\xi$.
\end{proof}

\section{Proof of Theorem \ref{main}}

The possibility of replacing $t\mapsto t^k/(k!\mu^k)$ with $V_k$ is justified by Lemma \ref{lem:vk}(c).

Since $Y_1$ is a renewal process, the case $k=1$ of Theorem \ref{main} was known, see Proposition~3.5 in \cite{Iksanov+Jedidi+Bouzzefour:2017}. Thus, in what follows it is tacitly assumed that $k\geq 2$.

Throughout the proof, for notational simplicity, we assume that if the distribution of $\xi$ is lattice, its lattice span is $1$. Using \eqref{eq:recur} we obtain a basic decomposition for the present proof: for $k\geq 2$ and $t\geq 0$:
\begin{multline*}
Y_k(t)-V_k(t)=\sum_{r\geq 1}\big(Y_{k-1}^{(r)}(t-S_r)-V_{k-1}(t-S_r)\big)\1_{\{S_r\leq t\}}+
\left( \sum_{r\geq 1}V_{k-1}(t-S_r)\1_{\{S_r\leq t\}}-V_k(t) \right)
\\=:I_k(t)+J_k(t).
\end{multline*}
We shall prove that
\begin{equation}\label{eq:jt}
C\bigg(\bigg(\frac{a_k J_k(t)}{(2t^{2k-1}\log\log t)^{1/2}}: t>\eee\bigg)\bigg)=[-1,1]\quad\text{{\rm a.s.}}
\end{equation}
and that
\begin{equation}\label{eq:it}
\lim_{t\to\infty}\frac{I_k(t)}{(t^{2k-1}\log\log t)^{1/2}}=0\quad\text{a.s.},
\end{equation}
that is, the term $J_k$ gives the principal contribution, whereas the contribution of $I_k$ is negligible.

First, we deal with \eqref{eq:jt}. Recalling \eqref{eq:recurV} write, with the help of integration by parts, for $k\geq 2$ and $t\geq 0$,
\begin{eqnarray*}
&&J_k(t)=\int_{[0,\,t]}V_{k-1}(t-x){\rm d}(Y(x)-V(x))=\int_{[0,\,t)}(Y(t-x)-V(t-x)){\rm
d}V_{k-1}(x)\\&=&\int_{[0,\,t)}(Y(t-x)-V(t-x)-\sigma\mu^{-3/2}W(t-x)){\rm
d}V_{k-1}(x)\\&+&\sigma\mu^{-3/2}\int_{[0,\,t)}W(t-x){\rm
d}V_{k-1}(x)\bigg)=: A_k(t)+\sigma\mu^{-3/2}B_k(t),
\end{eqnarray*}
where $W$ is a standard Brownian motion appearing in Lemma \ref{chs}. By Lemmas \ref{chs} and~\ref{lem:vk}(a),
\begin{multline*}
|A_k(t)|\leq\sup_{0\leq u\leq t}|Y(u)-V(u)-\sigma\mu^{-3/2}W(u)|V_{k-1}(t)\\=o\big((t^{2k-1}\log\log t)^{1/2}\big),\quad t\to\infty\quad\text{a.s.}
\end{multline*}
Further,
\begin{multline*}
B_k(t)=\frac{1}{(k-1)!\mu^{k-1}}\int_{(0,\,t]}(t-x)^{k-1}{\rm
d}W(x)+\int_{(0,\,t]}\Big(V_{k-1}(t-x)-\frac{(t-x)^{k-1}}{(k-1)!\mu^{k-1}}\Big){\rm
d}W(x)\\=:((k-1)!\mu^{k-1})^{-1}B_{1,k}(t)+B_{2,k}(t).
\end{multline*}
We intend to prove that $$\lim_{t\to\infty}\frac{B_{2,k}(t)}{t^{k-1/2}}=0\quad\text{a.s.}$$
To this end, it suffices to show that, for all $\varepsilon>0$,
\begin{equation}\label{aim1}
\sum_{n\geq 1}\mmp\{\sup_{t\in [n,\, n+1]}\,|B_{2,k}(t)|>\varepsilon n^{k-1/2}\}<\infty.
\end{equation}
Indeed, if this is true, then, by the Borel–Cantelli lemma, $$\sup_{t\in [n,\, n+1]}|B_{2,k}(t)|\leq\varepsilon n^{k-1/2}$$ for $n$ large enough a.s. Hence, for all large enough $n$ and $t\in [n,\, n+1]$, $$|B_{2,k}(t)|\leq \sup_{t\in [n,\,n+1]}|B_{2,k}(t)|\leq \varepsilon n^{k-1/2}\leq \varepsilon t^{k-1/2}\quad\text{a.s.}$$ Thus, ${\lim\sup}_{t\to \infty}|B_{2,k}(t)|/t^{k-1/2}\leq \varepsilon$ a.s.\ which entails the claim.

Let us prove~\eqref{aim1}. In what follows $C_1$, $C_2,\ldots$ will denote positive constants, whose values are of no importance. Put $f_k(t):=V_{k-1}(t)- ((k-1)!\mu^{k-1})^{-1} t^{k-1}$ for $k\geq 2$ and $t\geq 0$. Write
\begin{multline*}
\sup_{t\in [n,\,n+1]}|B_{2,k}(t)-B_{2,k}(n)|=\sup_{t\in [0,\,1]}\Big|\int_{(0,\,n+t]} f_k(n+t-x){\rm d}W(x)-\int_{(0,\,n]} f_k(n-x){\rm d}W(x)\Big|\\=\sup_{t\in [0,\,1]}\Big|\int_{(n,\,n+t]}f_k(n+t-x){\rm d}W(x)+\int_{(0,\,n]}(f_k(n+t-x)-f_k(n-x)){\rm d}W(x)\Big|\\ \leq \sup_{t\in [0,\,1]}\Big|\int_{(n,\,n+t]}f_k(n+t-x){\rm d}W(x)\Big|+\sup_{t\in [0,\,1]}\Big|\int_{(0,\,n]}(f_k(n+t-x)-f_k(n-x)){\rm d}W(x)\Big|.
\end{multline*}
Note that the variable $B_{2,k}(n)$ has a normal distribution with zero mean and variance $\int_0^n (f_k(x))^2{\rm d}x$. By Lemma \ref{lem:vk}(c), for large enough $n$, $\int_0^n (f_k(x))^2{\rm d}x\leq C_1 n^{2k-3}$. Hence, for all $\varepsilon>0$ and large $n$, $$\mmp\{|B_{2,k}(n)|>\varepsilon n^{k-1/2}\}\leq \Big(\frac{2}{\pi}\Big)^{1/2} \int_{\varepsilon C_1^{-1/2} n }^\infty \eee^{-x^2/2}{\rm d}x\leq \Big(\frac{2C_1}{\varepsilon^2 \pi}\Big)^{1/2}\frac{\eee^{-\varepsilon^2 n^2/(2C_1)}}{n}.$$ The right-hand side is the $n$th term of a summable sequence.

Observe that the process $B_{2,k}$ is a.s.\ continuous. Indeed, $$B_{2,k}(t)=\int_{[0,\,t)}W(t-x){\rm
d}V_{k-1}(x)-\frac{1}{(k-1)!\mu^{k-1}}\int_{[0,\,t)}W(t-x){\rm
d}x^{k-1},$$ and each of the summands is a.s.\ continuous as the Lebesgue-Stieltjes convolution of an a.s.\ continuous function and nondecreasing function. In view of the a.s.\ continuity, which entails the a.s.\ boundedness on $[0,1]$, we infer, for all $\varepsilon>0$,
\begin{equation}\label{part}
-\log \mmp\{\sup_{t\in [0,\,1]}\,|B_{2,k}(t)|>\varepsilon n^{k-1/2}\}~\sim~ \frac{\varepsilon^2 n^{2k-1}}{2\int_0^1 f_k^2(y){\rm d}y},\quad n\to\infty
\end{equation}
by a large deviation bound for a.s.\ bounded Gaussian processes, see formula (1.1) in \cite{Marcus+Shepp:1972}. Since the variable $\sup_{t\in [0,\,1]}\,\Big|\int_{(n,\,n+t]}f_k(n+t-x){\rm d}W(x)\Big|$ has the same distribution as $\sup_{t\in [0,\,1]}\,|B_{2,k}(t)|$ we conclude that the sequence $$n\mapsto \mmp\Big\{\sup_{t\in [0,\,1]}\,\Big|\int_{(n,\,n+t]}f_k(n+t-x){\rm d}W(x)\Big|>\varepsilon n^{k-1/2}\Big\}$$ is summable.

To proceed, we note that the variable $\sup_{t\in [0,\,1]}\Big|\int_{(0,\,n]}(f_k(n+t-x)-f_k(n-x)){\rm d}W(x)\Big|$ has the same distribution as $\sup_{t\in [0,\,1]}\Big|\int_{(0,\,n]}(f_k(x+t)-f_k(x)){\rm d}W(x)\Big|$.  Whenever a Skorokhod integral is well-defined, it coincides with the result of (formal) integration by parts. In particular,
\begin{multline*}
\int_{(0,\,n]}(f_k(x+t)-f_k(x)){\rm d}W(x)=(f_k(n+t)-f_k(n))W(n)-\int_{(0,\,n]}W(x){\rm d}_x(f_k(x+t)-f_k(x))\\=(f_k(n+t)-f_k(n))W(n)+\int_{(0,\,t]}W(x){\rm d}f_k(x)+\int_{(0,\,n-t]}(W(x+t)-W(x)){\rm d}f_k(x+t)\\-\int_{(n-t,\,n]}W(x){\rm d}f_k(x+t).
\end{multline*}
Hence, since the function $V_{k-1}$ is nondecreasing,
\begin{multline}
\sup_{t\in [0,\,1]}\,\Big|\int_{(0,\,n]}(f_k(x+t)-f_k(x)){\rm d}W(x)\Big|\leq (V_{k-1}(n+1)-V_{k-1}(n)) |W(n)|\\+((k-1)!\mu^{k-1})^{-1} ((n+1)^{k-1}-n^{k-1})|W(n)|+\sup_{t\in [0,\,1]}\Big|\int_{(0,\,t]}W(x){\rm d}V_{k-1}(x)\Big|\\+((k-1)!\mu^{k-1})^{-1}\sup_{t\in [0,\,1]}\Big|\int_{(0,\,t]}W(x){\rm d}x^{k-1}\Big|+\sup_{t\in [0,\,1]}\, \int_{(0,\,n-t]}|W(x+t)-W(x)|{\rm d}_xV_{k-1}(x+t)\\+((k-1)!\mu^{k-1})^{-1}\sup_{t\in [0,\,1]}\, \int_{(0,\,n-t]}|W(x+t)-W(x)|{\rm d}_x(x+t)^{k-1}\\+\sup_{t\in [0,\,1]}\, \int_{(n-t,\,n]}|W(x)|{\rm d}_x V_{k-1}(x+t)+((k-1)!\mu^{k-1})^{-1}\sup_{t\in [0,\,1]}\, \int_{(n-t,\,n]}|W(x)|{\rm d}_x (x+t)^{k-1}.\label{inter}
\end{multline}

We shall only treat the terms involving $V_{k-1}$, for the analysis of the terms involving $t\mapsto t^{k-1}$ is analogous but easier. We start with the penultimate term in \eqref{inter}. Observe that
\begin{multline*}
\sup_{t\in [0,\,1]}\, \int_{(n-t,\,n]}|W(x)|{\rm d}_x V_{k-1}(x+t)\leq \sup_{t\in [0,\,1]}\,(V_{k-1}(n+t)-V_{k-1}(n))\sup_{n-t\leq z\leq n}\,|W(z)|\\\leq (V_{k-1}(n+1)-V_{k-1}(n)) \sup_{n-1\leq z\leq n}\,|W(z)|.
\end{multline*}
By Lemma \ref{lem:vk}(b), for large $n$, $V_{k-1}(n+1)-V_{k-1}(n)\leq C_2 n^{k-2}$. The random variable $\sup_{z\in [n-1,\,n]}\,|W(z)|$ has the same distribution as $\sup_{z\in [0,1]}\,|W(z)+W^\prime(n-1)|$, where $W^\prime(n-1)$ is a copy of $W(n-1)$ which is independent of $\sup_{z\in [0,1]}\,|W(z)|$. Hence, for all $\varepsilon>0$ and large $n$,
\begin{multline*}
\mmp\{(V_{k-1}(n+1)-V_{k-1}(n))\sup_{z\in [n-1,\,n]}\,|W(z)|>\varepsilon n^{k-1/2}\}\\\leq \mmp\{(V_{k-1}(n+1)-V_{k-1}(n))\sup_{z\in [0,1]}\,|W(z)|>\varepsilon n^{k-1/2}/2\}\\+\mmp\{(V_{k-1}(n+1)-V_{k-1}(n))|W^\prime(n-1)|>\varepsilon n^{k-1/2}/2\}=:R_{n,1}+R_{n,2}.
\end{multline*}
Further, $$R_{n,2}\leq \Big(\frac{2}{\pi}\Big)^{1/2} \int_{2^{-1}C_2^{-1}\varepsilon n }^\infty \eee^{-x^2/2}{\rm d}x \\\leq \Big(\frac{8C_2^2}{\varepsilon^2 \pi}\Big)^{1/2}\frac{\eee^{-\varepsilon^2 n^2/(8C_2^2)}}{n}.$$ The right-hand side is the $n$th term of a summable sequence. Using the inequalities
\begin{equation}\label{ineq}
\mmp\{\sup_{t\in [0,\,1]}\,|W(t)|>x\}\leq 2\mmp\{\sup_{t\in [0,\,1]}\,W(t)>x\}=2\mmp\{|W(1)|>x\},\quad x>0
\end{equation}
we also conclude that the sequence $(R_{n,1})_{n\in\mn}$ is summable. Thus, for all $\varepsilon>0$, the sequence $$n\mapsto \mmp\{(V_{k-1}(n+1)-V_{k-1}(n))\sup_{z\in [n-1,\,n]}\,|W(z)|>\varepsilon n^{k-1/2}\}$$ is summable, and so is $$n\mapsto \mmp\{(V_{k-1}(n+1)-V_{k-1}(n))|W(n)|>\varepsilon n^{k-1/2}\}$$ because $|W(n)|\leq \sup_{z\in [n-1,\,n]}\,|W(z)|$. For all $\varepsilon>0$, the sequence $$n\mapsto \mmp\Big\{\sup_{t\in [0,\,1]}\Big|\int_{(0,\,t]}W(x){\rm d}V_{k-1}(x)\Big|>\varepsilon n^{k-1/2}\Big\}$$ is summable in view of the bound $$\sup_{t\in [0,\,1]}\Big|\int_{(0,\,t]}W(x){\rm d}V_{k-1}(x)\Big|\leq \int_{(0,\,1]}|W(x)|{\rm d}V_{k-1}(x)\leq \sup_{t\in [0,\,1]}\,|W(t)|V_{k-1}(1)$$ and \eqref{ineq}.

Finally, $$\sup_{t\in [0,\,1]}\, \int_{(0,\,n-t]}|W(x+t)-W(x)|{\rm d}_xV_{k-1}(x+t)\leq V_{k-1}(n) \sup_{t\in [0,\,1]}\,\sup_{z\in [0,\,n]}|W(z+t)-W(z)|.$$ By Lemma \ref{lem:vk}(a), $V_{k-1}(n)\leq C_3 n^{k-1}$ for large $n$. By Lemma 1.2.1 on p.~29 in \cite{Csorgo+Revesz:1981}, given $\delta>0$ there exists $C=C(\delta)>0$ such that, for all $\varepsilon>0$ and $n\geq 2$, $$\mmp\{V_{k-1}(n)\sup_{t\in [0,\,1]}\,\sup_{z\in [0,\,n]}|W(z+t)-W(z)| >\varepsilon
n^{k-1/2}\}\leq C(n+1)\exp\Big(-\frac{\varepsilon^2 n}{C_3^2(2+\delta)}\Big).$$ This proves that the sequence $$n\mapsto \mmp\Big\{\sup_{t\in [0,\,1]}\, \int_{(0,\,n-t]}|W(x+t)-W(x)|{\rm d}_xV_{k-1}(x+t)>\varepsilon n^{k-1/2}\Big\}$$ is summable. Combining fragments together we arrive at \eqref{aim1}.

We are now in position to prove~\eqref{eq:jt}. By Theorem 1 in \cite{Lachal:1997}, $${\lim\sup}_{t\to\infty}\frac{(2k-1)^{1/2}B_{1,k}(t)}{(2t^{2k-1}\log\log t)^{1/2}}=1\quad\text{a.s.}$$ Since $-W$ is also a Brownian motion we infer $${\lim\inf}_{t\to\infty}\frac{(2k-1)^{1/2}B_{1,k}(t)}{(2t^{2k-1}\log\log t)^{1/2}}=-1\quad\text{a.s.}$$ and thereupon $$C\bigg(\bigg(\frac{(2k-1)^{1/2}B_{1,k}(t)}{(2t^{2k-1}\log\log t)^{1/2}}: t>\eee\bigg)\bigg)=[-1,1]\quad\text{{\rm a.s.}}$$ because the random function $t\mapsto B_{1,k}(t) t^{1/2-k}(\log\log t)^{-1/2}$ is a.s.\ continuous on $(\eee,\infty)$. This completes the proof of \eqref{eq:jt}.

Now we pass to a proof of \eqref{eq:it}. Recall that $k\geq 2$.
Invoking Lemmas \ref{lem:vk}(a) and \ref{center:mom} yields
\begin{multline}
\me (I_k(t))^2=
\int_{[0,\,t]} \me\big(Y_{k-1}(t-x)-V_{k-1}(t-x)\big)^2 \dd V(x)\\
\le \me \big(\sup_{s\in [0,\,t]} (Y_{k-1}(s)-V_{k-1}(s))^2\big) \cdot V(t) = O(t^{2k-2}), \quad t\to\infty. \label{eq:mom}
\end{multline}
By Markov's inequality and \eqref{eq:mom}, for all $\varepsilon>0$,
$$
\sum_{n\geq 1} \mmp\Big\{ \frac{I_k(n^{3/2})}{n^{(3/2)(k-1/2)}}>\varepsilon \Big\} \le \sum_{n\geq 1} \frac{\me (I_k(n^{3/2}))^2}{\varepsilon^2 n^{3(k-1/2)}}
< \infty.
$$
Hence,
\begin{equation}\label{eq:xkn}
	\lim_{n\to\infty} \frac{I_k(n^{3/2})}{n^{(3/2)(k-1/2)}} = 0 \quad \text{a.s.}
\end{equation}
It remains to pass from an integer argument to a continuous argument. For any $t\geq 0$, there exists $n\in\mn_0$ such that
$t\in [n^{3/2}, (n+1)^{3/2})$. By monotonicity,
\begin{multline*}
 \frac{I_k(t)}{t^{k-1/2}}\le \frac{I_k((n+1)^{3/2})}{n^{(3/2)(k-1/2)}}\\ + \frac{\int_{[0,\,(n+1)^{3/2}]}V_{k-1}((n+1)^{3/2}-x){\rm d}Y(x)-\int_{[0,\,n^{3/2}]}V_{k-1}(n^{3/2}-x){\rm d}Y(x)}{n^{(3/2)(k-1/2)}}.
\end{multline*}
Relation \eqref{eq:xkn} implies that the first summand on the right-hand side converges to $0$ a.s.\ as $n\to\infty$. The second summand is equal to
\begin{multline*}
\int_{(n^{3/2},\,(n+1)^{3/2}]}V_{k-1}((n+1)^{3/2}-x){\rm d}Y(x)+\int_{[0,\,n^{3/2}]}(V_{k-1}((n+1)^{3/2}-x)-V_{k-1}(n^{3/2}-x)){\rm d}Y(x)\\=:X_{k,1}(n)+X_{k,2}(n).
\end{multline*}
By monotonicity,
\begin{multline*}
X_{k,1}(n)\le V_{k-1}((n+1)^{3/2}-n^{3/2})\big(Y((n+1)^{3/2})-Y(n^{3/2})\big)\\=o(n^{k/2+1})=o(n^{(3/2)(k-1/2)}),\quad n\to\infty\quad\text{a.s.}
\end{multline*}
Here, the penultimate equality is justified by the strong law of large numbers for renewal processes $\lim_{n\to\infty}n^{-1}Y(n)=\mu^{-1}$ a.s.\ and $V_{k-1}((n+1)^{3/2}-n^{3/2})=O(n^{(k-1)/2})$ as $n\to\infty$ which holds true by Lemma \ref{lem:vk}(a).

Using Lemma \ref{lem:vk}(d) we infer
\begin{multline*}
X_{k,2}(n)\leq \big(V((n+1)^{3/2}-n^{3/2})+1\big)(V((n+1)^{3/2}))^{k-2}Y(n^{3/2})=O(n^{(3/2)(k-2/3)})\\=o(n^{(3/2)(k-1/2)}),\quad n\to\infty\quad\text{a.s.}
\end{multline*}
The penultimate equality is secured by Lemma \ref{lem:vk}(a) and the strong law of large numbers for renewal processes.

We have shown that $$\lim\sup_{t\to\infty}t^{-(k-1/2)}I_k(t)\leq 0\quad\text{a.s.}$$ An analogous argument proves the converse inequality for the lower limit.
The proof of Theorem \ref{main} is complete.

\section{Appendix}

Lemma \ref{lattice} is a lattice analogue of Theorem 2.2 in \cite{Iksanov+Rashytov+Samoilenko:2023} dealing with iterated perturbed random walks. In the proof of Lemma \ref{lem:vk}(c) we only need a version of Lemma \ref{lattice} for iterated standard random walks.
\begin{lemma}\label{lattice}
Let $d>0$. Assume that the distributions of $\xi$ and $\eta$ are $d$-lattice and that $\me \xi^2<\infty$ and $\me \eta <\infty$. Then, for each fixed $k\in\mn$,
\begin{equation}\label{eq:lattice}
V^\ast_k(nd)-\frac{(nd)^k}{k!\mu^k}~\sim~ \frac{(nd)^{k-1}\Big(\frac{d}{2\mu}\big(2k-1\big)+k\big(\frac{\me \xi^2}{2\mu^2}-\frac{\me \eta}{\mu}\big)\Big)}{\mu^{k-1}(k-1)!},\quad n\to\infty,
\end{equation}
where $\mu=\me\xi<\infty$ and $V^\ast_k$ is a counterpart of $V_k$ for iterated perturbed random walks.
\end{lemma}
\begin{proof}
We use mathematical induction in $k$. Let $k=1$.
Put $U(t):=V(t)+1$ for $t\geq 0$, so that $U$ is the renewal function.
Since $V_1^\ast(t)=\sum_{j\geq 1}\mmp\{S_{j-1}+\eta_j\leq t\}=:V^\ast(t)$ for $t\geq 0$, we infer, for $n\in\mn$,
\begin{multline*}
V^\ast(nd)-\frac{nd}{\mu}=\int_{[0,\,nd]}\Big(U(nd-x)-\frac{nd-x}{\mu}\Big){\rm d}\mmp\{\eta\leq x\}-\frac{1}{\mu}\int_0^{nd}\mmp\{\eta>x\}{\rm d}x\\= \sum_{r=1}^n \Big(U((n-r)d)-\frac{(n-r)d}{\mu}\Big)\mmp\{\eta=rd\}-\frac{1}{\mu}\int_0^{nd}\mmp\{\eta>x\}{\rm d}x.
\end{multline*}
According to formula (5.14) on p.~59 in \cite{Gut:2009}, $$\lim_{n\to\infty} \Big(U(nd)-\frac{nd}{\mu}\Big)= \frac{d}{2\mu}+\frac{\me \xi^2}{2\mu^2}=:D.$$
Hence, given $\varepsilon>0$ there exists $n_0\in\mn$ such that $U(nd)-\mu^{-1}nd \le D+\varepsilon$ for all $n\ge n_0$. With this at hand, for $n\geq n_0$,
\begin{multline*}
		\sum_{r=1}^n \Big(U((n-r)d)-\frac{(n-r)d}{\mu}\Big)\mmp\{\eta=rd\}
=\sum_{r=1}^{n-n_0}\ldots+\sum_{r=n-n_0+1}^n\ldots\\\le D+\varepsilon
+\sup_{1\le r\le n_0-1} \Big(U(rd)-\frac{rd}{\mu}\Big)\mmp\{\eta\geq (n-n_0+1)d\}.
	\end{multline*}
Letting $n\to\infty$ and then $\varepsilon\to 0+$ we conclude that $${\lim\sup}_{n\to\infty}\sum_{r=1}^n \Big(U((n-r)d)-\frac{(n-r)d}{\mu}\Big)\mmp\{\eta=rd\}\leq D.$$ The converse inequality for the lower limit follows analogously. Noting that $\lim_{n\to\infty}\int_0^{nd}\mmp\{\eta>x\}{\rm d}x=\me \eta$ completes the proof of \eqref{eq:lattice} with $k=1$.
	
Assume now that \eqref{eq:lattice} holds for $k\leq j$. In particular, given $\varepsilon>0$ there exists $n_0\in\mn$ such that $$V_k^\ast(nd)-\frac{(nd)^k}{k!\mu^k}\le (C_k+\varepsilon)\frac{(nd)^{k-1}}{(k-1)!\mu^{k-1}},\quad 1\leq k\leq j,$$ where $$C_k:=\frac{d}{2\mu}\big(2k-1\big)+k\Big(\frac{\me \xi^2}{2\mu^2}-\frac{\me \eta}{\mu}\Big),\quad k\in\mn.$$ Recalling \eqref{eq:recurV} we obtain
\begin{multline*}
V^\ast_{j+1}(nd)-\frac{(nd)^{j+1}}{(j+1)!\mu^{j+1}}=\sum_{r=1}^n \Big(V^\ast((n-r)d)-\frac{(n-r)d}{\mu}\Big)\Big(V^\ast_j(rd)-V^\ast_j((r-1)d)\Big)\\+\frac{d}{\mu}\sum_{r=1}^{n-1}
\Big(V_j^\ast(rd)- \frac{(rd)^j}{j!\mu^j}\Big)+\frac{d^{j+1}}{j!\mu^{j+1}}\Big(\sum_{r=1}^{n-1}r^j-\frac{n^{j+1}}{j+1}\Big).
	\end{multline*}
Hence, for $n\geq n_0$,
\begin{multline*}
	A_j(n):=\sum_{r=1}^n \Big(V^\ast((n-r)d)-\frac{(n-r)d}{\mu}\Big)\Big(V^\ast_j(rd)-V^\ast_j((r-1)d)\Big)=\sum_{r=1}^{n-n_0}\ldots+\sum_{r=n-n_0+1}^n\ldots\\\le (C_1+\varepsilon) V^\ast_j((n-n_0)d)+\sup_{1\le r\le n_0-1} \Big(V^\ast(rd)-\frac{rd}{\mu}\Big)\big(V^\ast_j(nd)-V^\ast_j((n-n_0)d)\big).
\end{multline*}
Invoking parts (a) and (b) of Lemma \ref{lem:vk} and letting $n\to\infty$ and then $\varepsilon\to 0+$ yields
$$
\limsup_{n\to\infty} \frac{A_j(n)}{(nd)^j}\le \frac{C_1}{j!\mu^j}.
$$
Further,
\begin{multline*}
B_j(n):=\sum_{r=1}^{n-1}\Big(V_j^\ast(rd)- \frac{(rd)^j}{j!\mu^j}\Big)=\sum_{r=1}^{n_0-1}\ldots+\sum_{r=n_0}^n\ldots \le n_0 \sup_{1\le r\le n_0-1} \Big(V^\ast_j(rd)-\frac{(rd)^j}{j!\mu^j}\Big)\\+(C_j+\varepsilon)\sum_{r=n_0}^n \frac{(rd)^{j-1}}{(j-1)!\mu^{j-1}}=(C_j+\varepsilon)\frac{d^{j-1}}{(j-1)!\mu^{j-1}}\frac{n^j}{j}+o(n^j),\quad n\to\infty
\end{multline*}
having utilized Faulhaber's formula for the last equality. Thus, $${\lim\sup}_{n\to\infty} \frac{d}{\mu}\frac{B_j(n)}{(nd)^j}\leq \frac{C_j}{j!\mu^j}.$$ Analogous arguments prove the converse inequalities for the lower limits involving both $A_j(n)$ and $B_j(n)$.
Finally, $$\frac{d^{j+1}}{j!\mu^{j+1}}\Big(\sum_{r=1}^{n-1}r^j-\frac{n^{j+1}}{j+1}\Big)~\sim~ \frac{d}{2\mu}\frac{(nd)^j}{j!\mu^j},\quad n\to\infty$$ by another application of Faulhaber's formula.

Combining all the fragments together we conclude that
$$
V^\ast_{j+1}(nd)-\frac{(nd)^{j+1}}{(j+1)!\mu^{j+1}}~\sim~\Big(C_1+C_j+\frac{d}{2\mu}\Big)\frac{(nd)^j}{j!\mu^j}=C_{j+1}\frac{(nd)^j}{j!\mu^j},\quad n\to\infty.
$$
\end{proof}

\end{document}